\documentclass{amsart}
\usepackage{amsmath}
\usepackage{amssymb}
\usepackage{amsthm}

\newtheorem{theorem}{Theorem}

\newtheorem{lemma}{Lemma}

\newtheorem{exmp}{Example}
\newtheorem{cor}{Corollary}

\begin{document}
\title{Orthogonality preserving transformations of Hilbert Grassmannians}
\author{Mark Pankov}
\subjclass[2000]{47N50, 81P10, 81R15}

\keywords{complex Hilbert space, Hilbert Grassmannian, orthogonality preserving transformations, 
Wigner-type theorem}
\address{Faculty of Mathematics and Computer Science, 
University of Warmia and Mazury, S{\l}oneczna 54, Olsztyn, Poland}
\email{pankov@matman.uwm.edu.pl}

\maketitle

\begin{abstract}
Let $H$ be a complex Hilbert space and let ${\mathcal G}_{k}(H)$ be the Grassmannian formed by
$k$-dimensional subspaces of $H$. 
Suppose that $\dim H>2k$ and $f$ is an orthogonality preserving injective transformation of ${\mathcal G}_{k}(H)$,
i.e. for any orthogonal $X,Y\in {\mathcal G}_{k}(H)$ the images $f(X),f(Y)$ are orthogonal. 
Furthermore, if $\dim H=n$ is finite, then $n=mk+i$ for some integers $m\ge 2$ and $i\in \{0,1,\dots,k-1\}$
(for $i=0$ we have $m\ge 3$).
We show that $f$ is a bijection induced by a unitary or anti-unitary operator 
if $i\in \{0,1,2,3\}$ or $m\ge i+1\ge 5$; in particular,
the statement holds for $k\in \{1,2,3,4\}$ and, if $k\ge 5$,
then there are precisely $(k-4)(k-3)/2$ values of $n$ such that the above condition is not satisfied.
As an application, we obtain a result concerning the case when $H$ is infinite-dimensional.
\end{abstract}

\section{Introduction}
By Gleason's theorem \cite{Gleason},  
the set of pure states of a quantum mechanical system can be identified with
the set of rank-one projections, i.e. the set of rays in a complex Hilbert space. 
Wigner's theorem \cite{Wigner} describes symmetries of quantum mechanical systems,
it  states that every transformation of the set of pure states preserving the transition probability 
(the trace of the composition of two projections or, equivalently, the angle between two rays) 
is induced by a linear or conjugate-linear isometry. 
Various kinds of Wigner-type theorems can be found, for example, in \cite{Pankov-book};
some of them are formulated in terms of orthogonality preserving transformations.

Let $H$ be a complex Hilbert space. 
For every positive integer $k<\dim H$ we denote by ${\mathcal G}_{k}(H)$ the Grassmannian formed by 
$k$-dimensional subspaces of $H$.
This Grassmannian can be naturally identified with the set of rank-$k$ projections. 
In the case when $\dim H\ge 2k$, two $k$-dimensional subspaces are orthogonal 
if and only if the composition of the corresponding projections is zero.

Suppose that $\dim H\ge 3$.
Then the bijective version of Wigner's theorem is a consequence of the following Uhlhorn's  observation \cite{Uhlhorn}:
every bijective transformation of ${\mathcal G}_{1}(H)$
preserving the orthogonality relation in both directions is induced by a unitary or anti-unitary operator.
In fact, the latter statement is a reformulation of the Fundamental Theorem of Projective Geometry
(see \cite[Proposition 4.8]{Pankov-book} or \cite[Theorem 4.29]{Var}).

Uhlhorn's  observation was extended on other Grassmannians by Gy\H{o}ry \cite{Gyory} and \v{S}emrl \cite{Semrl}: 
if $\dim H>2k$, then every bijective transformation of ${\mathcal G}_{k}(H)$
preserving the orthogonality relation in both directions is induced by a unitary or anti-unitary operator.
 A simple example shows that the statement fails for $\dim H=2k$. 
If $H$ is infinite-dimensional, then the same holds for orthogonality preserving (in both directions)
bijective transformations of the Grassmannian formed by subspaces whose dimension and codimension both are infinite
\cite{Semrl}.
Gy\H{o}ry--\v{S}emrl's theorem is used to study transformations preserving the gap metric \cite{GS}
and commutativity preserving transformations \cite{Pankov-c1,Pankov-c2}. 
The assumption of surjectivity cannot be omitted.
It was noted in \cite{Semrl} that for the case when $H$ is infinite-dimensional
there are non-surjective transformations of ${\mathcal G}_{k}(H)$ which are not induced by linear or conjugate-linear isometries 
and preserve the orthogonality relation in both directions.  
If $H$ is finite-dimensional and $\dim H>2k$, then  such transformations do not exist, i.e. 
every transformation of ${\mathcal G}_{k}(H)$ preserving the orthogonality relation in both directions
is a bijection induced by a unitary or anti-unitary operator \cite{Pankov1}.

In this note, we obtain analogues of the above mentioned results for injective transformations
under the assumption that the orthogonality relation is preserved exactly in one direction.

\section{Results}
We start from the case when $H$ is finite-dimensional.
Let $k$ be a positive integer such that $\dim H>2k$.
Then $$\dim H=mk+i$$ for some integers $m\ge 2$ and $i\in \{0,1,\dots,k-1\}$.
Note that $m\ge 3$ if $i=0$.

\begin{theorem}\label{theorem1}
Suppose that one of the following conditions is satisfied:
\begin{enumerate}
\item[$\bullet$] $i\in \{0,1,2,3\}${\rm;}
\item[$\bullet$] $i\ge 4$ and $m\ge i+1$.
\end{enumerate}
Then every injective transformation $f$ of ${\mathcal G}_{k}(H)$ preserving the orthogonality relation,
i.e for any orthogonal $X,Y\in {\mathcal G}_{k}(H)$ the images $f(X),f(Y)$ are orthogonal,
is a bijection induced by a unitary or anti-unitary operator on $H$.
\end{theorem}
Theorem \ref{theorem1} shows that 
every injective transformation of ${\mathcal G}_{k}(H)$ preserving the orthogonality relation
is a bijection induced by a unitary or anti-unitary operator if 
\begin{enumerate}
\item[$\bullet$] $k\in \{1,2,3,4\}$ or
\item[$\bullet$] $k\ge 5$ and $\dim H$ is sufficiently large, for example, 
if $\dim H\ge (k-2)k$.
\end{enumerate}
For $k\ge 5$ an easy calculation shows that 
there are precisely 
$$(k-4)+(k-5)+\dots+1=(k-3)(k-4)/2$$ values of $\dim H$ such that the condition of Theorem \ref{theorem1}
does not hold.

If $k=1$, then  Theorem \ref{theorem1} is a simple consequence of
the Fundamental Theorem of Projective Geometry \cite[Remark 49]{Pankov-book}.

Now, we suppose that  $H$ is infinite-dimensional
and consider an example of orthogonality preserving (in both directions)
transformation of ${\mathcal G}_{k}(H)$ which is not induced by a linear or conjugate-linear isometry.

\begin{exmp}[\cite{Semrl}]\label{exmp1}{\rm
Let be $U:H\to H$ be a linear isometry whose image is a proper subspace of $H$.
Then the adjoint operator $U^*$ is surjective and its kernel is the orthogonal complement of the image of $U$. 
We fix $X\in {\mathcal G}_{k}(H)$. 
Since $U^{*}U$ is identity, there is $X'\in {\mathcal G}_{k}(H)$
which is not contained in the image of $U$ and such that $U^{*}(X') = X$.
Denote by $f$ the transformation of ${\mathcal G}_{k}(H)$ 
which sends $X$ to $X'$ and every $Y\in {\mathcal G}_{k}(H)\setminus \{X\}$ to $U(Y)$.
We obtain  an injection preserving the orthogonality relation in both directions
which is not induced by a linear or conjugate-linear isometry.
Let $H'$ be the smallest closed subspace of $H$ containing all elements from the image of $f$.
We take any $Y\in {\mathcal G}_{k}(H)\setminus \{X\}$ non-orthogonal to $X$
and any family $\{Z_{j}\}_{j\in J}$ of mutually orthogonal elements of ${\mathcal G}_{k}(H)$ whose sum is $Y^{\perp}$.
Then $f(Y)=U(Y)$ and $f(Z_j)=U(Z_j)$ for every $j\in J$.
Therefore, $U(H)$ is the orthogonal sum of $f(Y)$ and all $f(Z_j)$.
Note that $U(H)$ is a proper subspace of $H'$, since it does not contain $f(X)=X'$.
}\end{exmp}

\begin{cor}\label{cor}
Let $f$ be an injective transformation of ${\mathcal G}_{k}(H)$
and let $H'$ be the smallest closed subspace of $H$ containing all elements from the image of $f$.
Suppose that $f$ satisfies the following condition:
\begin{enumerate}
\item[(*)] for any orthogonal $X,Y\in {\mathcal G}_{k}(H)$ there  is a family $\{Z_{j}\}_{j\in J}$ 
of mutually orthogonal elements of ${\mathcal G}_{k}(H)$ such that each $Z_j$ is orthogonal to $X+Y$
and $H'$ is the orthogonal sum of $f(X),f(Y)$ and all $f(Z_j)$.
\end{enumerate}
Then $f$ is induced by a linear or conjugate-linear isometry.
\end{cor}

The condition (*) implies that $f$ is orthogonality preserving.
The transformation considered in Example \ref{exmp1} does not satisfy (*).

\section{Proof of Theorem \ref{theorem1}}

First of all, we present some facts which will be exploited to prove Theorem \ref{theorem1}.
Two $k$-dimensional subspaces of $H$ are called {\it adjacent} if 
their intersection is $(k-1)$-dimensional or, equivalently, their sum is $(k+1)$-dimensional.
Any two distinct $1$-dimensional subspaces of $H$ are adjacent.
Similarly, if $\dim H=n$ is finite, then any two distinct $(n-1)$-dimensional subspaces of $H$ are adjacent. 
%For the remaining cases the adjacency relation is non-trivial. 
If $X,Y\in {\mathcal G}_{k}(H)$, then the distance $d(X,Y)$ between $X$ and $Y$ is defined as
the smallest integer $m$ such that there is a sequence 
$$X=X_{0},X_{1},\dots,X_{m}=Y,$$
where $X_{j-1},X_{j}$ are adjacent elements of ${\mathcal G}_{k}(H)$ for every $j\in \{1,\dots, m\}$.
It is well known that 
$$d(X,Y)=k-\dim(X\cap Y)=\dim(X+Y)-k,$$
see, for example, \cite[Section 2.3]{Pankov-book}.

\begin{theorem}[\cite{Pankov1}]\label{P1}
Suppose that $\dim H>2k>2$ and $f$ is a transformation of ${\mathcal G}_{k}(H)$ satisfying the following 
conditions:
\begin{enumerate}
\item[$\bullet$] 
$f$ is adjacency preserving, i.e.  
for any adjacent $X,Y\in {\mathcal G}_{k}(H)$ the images $f(X),f(Y)$ are adjacent; 
\item[$\bullet$] $f$ is orthogonality preserving. 
%i.e. for any orthogonal $X,Y\in {\mathcal G}_{k}(H)$ the images $f(X),f(Y)$ are orthogonal.
\end{enumerate}
Then $f$ is induced by a linear or a conjugate-linear isometry.
\end{theorem}

Also, we will need the following result mentioned in the Introduction.

\begin{theorem}[\cite{Pankov1}]\label{P2}
If the dimension of $H$ is finite and greater than $2k$,
then every transformation of ${\mathcal G}_{k}(H)$ preserving the orthogonality relation in both directions 
is a bijection induced by a unitary or anti-unitary operator on $H$.
\end{theorem}

Suppose that $\dim H=n$ is finite. 
It was noted above that Theorem \ref{theorem1} holds for $k=1$.
We assume that  $k\ge 2$ and $$n=mk+i>2k,$$ where $m\ge 2$ and $i\in \{0,1,\dots,k-1\}$.
Let $f$ be an injective transformation of ${\mathcal G}_{k}(H)$ preserving the orthogonality relation.

\subsection{The case $i=0$}
In this case, we have $m\ge 3$.
Suppose that $f(X),f(Y)$ are orthogonal for some $X,Y\in {\mathcal G}_{k}(H)$;
we show that $X,Y$ are orthogonal. 

Observe that 
$$\dim(X^{\perp}\cap Y^{\perp})\ge n-2k=(m-2)k.$$
In the case when $m\ge 4$,
there are mutually orthogonal $k$-dimensional subspaces 
$$Z_{1},\dots,Z_{m-2}\subset X^{\perp} \cap Y^{\perp} ;$$
if $m=3$, then we take any $k$-dimensional subspace $Z_1$ in $X^{\perp} \cap Y^{\perp}$.
Since $f$ is orthogonality preserving,
each $f(Z_j)$ is orthogonal to $f(X)+f(Y)$.
By our assumption, $f(X)$ and $f(Y)$ are orthogonal.
The dimension of $H$ is equal to $mk$ and $H$ is the orthogonal sum of
$$f(Z_1),\dots,f(Z_{m-2}),f(X),f(Y).$$
The dimension of $X^{\perp}$ is equal to $n-k=(m-1)k$, i.e.
$X^{\perp}$ contains the unique $k$-dimensional subspace $Z$ orthogonal to all $Z_{j}$
and $H$ is the orthogonal sum of
$$f(Z_1),\dots,f(Z_{m-2}),f(X),f(Z).$$
Therefore,  $f(Y)=f(Z)$. 
Since $f$ is injective, we have $Y=Z$, i.e. $Y$ is orthogonal to $X$. 

So, $f$ is orthogonality preserving in both directions and Theorem \ref{P2} gives the claim.

\subsection{The case $i\in \{1,2,3\}$}
It is sufficient to show that $f$ is adjacency preserving and to apply Theorem \ref{P1}.

The general case can be reduced to the case when $m=2$.
If $X,Y\in {\mathcal G}_{k}(H)$ are adjacent and $m\ge 3$,
then we take mutually orthogonal $X_{1},\dots,X_{m-2}\in {\mathcal G}_{k}(H)$
whose sum is orthogonal to $X+Y$ (for $m=3$ we choose any $X_1\in {\mathcal G}_{k}(H)$ orthogonal to $X+Y$).
Consider the $(2k+i)$-dimensional subspaces $M$ and $N$ which are the orthogonal complements of
$$X_{1}+\dots+X_{m-2}\;\mbox{ and }\;f(X_{1})+\dots+f(X_{m-2}),$$
respectively. Then $f$ transfers any $k$-dimensional subspace of $M$ to a $k$-dimensional subspace of $N$,
i.e. it induces an orthogonality preserving injection of ${\mathcal G}_{k}(M)$ to ${\mathcal G}_{k}(N)$. 

From this moment, we assume that $m=2$, i.e. $n=2k+i$ with $i\in \{1,2,3\}$. 
Let ${\mathcal G}$ be the set of all subspaces $U\subset H$ satisfying 
$$k<\dim U \le k+i.$$
For every $U\in {\mathcal G}$
we define $f'(U)$ as the smallest subspace containing $f(Y)$ for all $k$-dimensional subspaces $Y\subset U$.
Since $f$ is injective, $\dim f'(U)>k$.
If $Z$ is a $k$-dimensional subspace orthogonal to $U$ 
(the inequality $\dim U^{\perp}\ge k$ guarantees that such a subspace exists),
then $f'(U)$ is orthogonal to $f(Z)$ which implies that $\dim f'(U)\le k+i$ and $f'(U)$ belongs to ${\mathcal G}$.
So, $f'$ is a transformation of ${\mathcal G}$ and for every $U\in {\mathcal G}$ we have 
$$f({\mathcal G}_{k}(U))\subset {\mathcal G}_{k}(f'(U)).$$
If $U,S\in {\mathcal G}$ are orthogonal, then the same holds for $f'(U),f'(S)$. 

The cases $i=1$ and $i=2$ are simple. 

\begin{lemma}\label{lemma-i=1,2}
If $i\in \{1,2\}$, then $f$ is adjacency preserving. 
\end{lemma}

\begin{proof}
If $i=1$, then ${\mathcal G}={\mathcal G}_{k+1}(H)$ and we have
$$f'({\mathcal G}_{k+1}(H))\subset {\mathcal G}_{k+1}(H)$$
which implies that $f$ is adjacency preserving
(since $X,Y\in {\mathcal G}_{k}(H)$ are adjacent if and only if $X+Y\in {\mathcal G}_{k+1}(H)$).

Let $i=2$. 
Then  $${\mathcal G}={\mathcal G}_{k+1}(H)\cup {\mathcal G}_{k+2}(H).$$
If $f$ is not adjacency preserving,
then  there is  $U\in {\mathcal G}_{k+1}(H)$ such that $f'(U)$ is an element of ${\mathcal G}_{k+2}(H)$.
Note that $U^{\perp}$ belongs to ${\mathcal G}_{k+1}(H)$ and $f'(U),f'(U^{\perp})$ are orthogonal.
This implies that the dimension of $f'(U^{\perp})$ is not greater than $k$
which is impossible.
\end{proof}

The case $i=3$ is more complicated and the proof will be given in several steps.
In this case, ${\mathcal G}$ is the union of 
${\mathcal G}_{k+1}(H)$, ${\mathcal G}_{k+2}(H)$ and ${\mathcal G}_{k+3}(H)$.
It is sufficient to establish that 
$$f'({\mathcal G}_{k+2}(H))\subset {\mathcal G}_{k+2}(H).$$ 
Indeed, if $U\in {\mathcal G}_{k+1}(H)$, then $U^{\perp}$ belongs to ${\mathcal G}_{k+2}(H)$
and the latter inclusion shows that  the same holds for $f'(U^{\perp})$. 
Since $f'(U)$ and $f'(U^{\perp})$ are orthogonal,
the dimension of $f'(U)$ is not greater than $k+1$ which means that $f'(U)\in {\mathcal G}_{k+1}(H)$.
So, $f'$ transfers ${\mathcal G}_{k+1}(H)$ to itself and $f$ is adjacency preserving. 

Our first step is to show that
$$f'({\mathcal G}_{k+2}(H))\subset {\mathcal G}_{k+1}(H)\cup {\mathcal G}_{k+2}(H).$$ 
If $U\in {\mathcal G}_{k+2}(H)$, then
$U^{\perp}\in {\mathcal G}_{k+1}(H)$ and $f'(U),f'(U^{\perp})$ are orthogonal.
In the case when $f'(U)$ belongs to ${\mathcal G}_{k+3}(H)$,
the dimension of $f'(U^{\perp})$ is not greater than $k$, a contradiction.

Two distinct elements of ${\mathcal G}_{k+1}(H)\cup{\mathcal G}_{k+2}(H)$ are 
said to be {\it$'$-adjacent} if one of the following possibilities is realized:
\begin{enumerate}
\item[$\bullet$] these are adjacent elements of ${\mathcal G}_{k+2}(H)$;
\item[$\bullet$] one of them belongs to ${\mathcal G}_{k+2}(H)$, the other to ${\mathcal G}_{k+1}(H)$
and the $(k+1)$-dimensional subspace is contained in the $(k+2)$-dimensional.
\end{enumerate}

\begin{lemma}\label{lemma-'}
If  $U,V\in {\mathcal G}_{k+2}(H)$ are adjacent, then 
$f'(U),f'(V)$ are $'$-adjacent or  $f'(U)=f'(V)$.
\end{lemma}

\begin{proof}
The subspace $U\cap V$ is $(k+1)$-dimensional and 
contains infinitely many $k$-dimensional subspaces. 
Therefore, the subspace $f'(U)\cap f'(V)$ contains infinitely many $k$-dimensional subspaces
which is possible only in the case when $f'(U),f'(V)$ are $'$-adjacent or coincident.
\end{proof}

The following lemma completes the proof for the case $i=3$. 

\begin{lemma}
$f'(U)\in {\mathcal G}_{k+2}(H)$ for every $U\in {\mathcal G}_{k+2}(H)$.
\end{lemma}

\begin{proof}
Recall that two subspaces $U,V\subset H$ are called {\it compatible} if 
there are subspaces $U'\subset U$ and $V'\subset V$ such that 
$U\cap V, U',V'$ are mutually orthogonal and 
$$U=U'+(U\cap V),\;\;\;V=V'+(U\cap V).$$
Let $U$ and $V$ be distinct $(k+2)$-dimensional subspaces of $H$.
Then $\dim(U\cap V)\ge 1$. The following two conditions are equivalent:
\begin{enumerate}
\item[$\bullet$] $U,V$ are compatible and $\dim(U\cap V)=1$;
\item[$\bullet$] there are infinitely many $k$-dimensional subspaces of $U$ which are orthogonal to
infinitely many $k$-dimensional subspaces of $V$.
\end{enumerate}
Suppose that one of these conditions holds.
Then there are infinitely many $k$-dimensional subspaces of $f'(U)$ which are orthogonal to
infinitely many $k$-dimensional subspaces of $f'(V)$.

The equation $\dim(U\cap V)=1$ shows that $d(U,V)=k+1$ and there is a sequence
$$U=U_{0},U_{1},\dots, U_{k},U_{k+1}=V,$$
where $U_{j},U_{j+1}$ are adjacent elements of ${\mathcal G}_{k+2}(H)$
for every $j\in \{0,1,\dots,k\}$.
By Lemma \ref{lemma-'}, 
$$f'(U_{j}),f'(U_{j+1})\in {\mathcal G}_{k+2}(H)\cup {\mathcal G}_{k+1}(H)$$ 
are $'$-adjacent or coincident;
in particular, if for a certain $j\in \{0,1,\dots,k\}$ both $f'(U_{j}),f'(U_{j+1})$ belong to ${\mathcal G}_{k+1}(H)$,
then $f'(U_j)=f'(U_{j+1})$.

Suppose that at least one of $f'(U),f'(V)$, say $f'(V)$,  belongs to ${\mathcal G}_{k+1}(H)$.
Since there are infinitely many $k$-dimensional subspaces of $f'(U)$ which are orthogonal to
infinitely many $k$-dimensional subspaces of $f'(V)$, one of the following possibilities is realized:
\begin{enumerate}
\item[(1)] $f'(U)\in {\mathcal G}_{k+2}(H)$ is the orthogonal complement of $f'(V)$;
\item[(2)] $f'(U),f'(V)\in {\mathcal G}_{k+1}(H)$ are orthogonal.
\end{enumerate}
In the case (1), we take the maximal  $j\in \{0,1,\dots,k+1\}$ such that $f'(U_j)$  belongs to ${\mathcal G}_{k+2}(H)$.
Since $f'(V)\in {\mathcal G}_{k+1}(H)$  (by assumption), we have $j\le k$.
Also, if $j<k$, then $$f'(U_{j+1})=\dots =f'(U_{k+1})=f'(V).$$
Therefore, $f'(V) \subset f'(U_j)$.
On the other hand, $f'(V)$ is the orthogonal complement of $f'(U)$ and
$$\dim (f'(U)\cap f'(U_j))=1.$$
The subspaces $U$ and $U_{j}$ are connected by the sequence $U=U_{0},U_{1},\dots,U_{j}$ with $j\le k$
which implies that 
$$\dim (f'(U)\cap f'(U_j))\ge k+2-j\ge 2,$$
and we get a contradiction.

In the case (2) is similar. We consider minimal $t$ and maximal  $j$ 
such that $f'(U_t)$ and $f'(U_j)$ belong to ${\mathcal G}_{k+2}(H)$.
Since $f'(U)$ and $f'(V)$ are elements of ${\mathcal G}_{k+1}(H)$, we have $t\ge 1$ and $j\le k$.
As in the previous case, we have
$$f'(U)=f'(U_{0})=\dots=f'(U_{t-1})\;\mbox{ and }\; f'(U_{j+1})=\dots =f'(U_{k+1})=f'(V)$$
if $t>1$ and $j<k$, respectively. 
Therefore, 
$$f'(U)\subset f'(U_t)\;\mbox{ and }\;f'(V)\subset f'(U_j).$$
Recall that  $f'(U),f'(V)$ are orthogonal $(k+1)$-dimensional subspaces.
This means that
$$\dim (f'(U_t)\cap f'(U_j))=1\mbox{ or } 2.$$
On the other hand, $U_{t}$ and $U_{j}$ are connected by 
a sequence of $j-t$ elements of ${\mathcal G}_{k+2}(H)$, where any two consecutive elements are adjacent
and $j-t\le k-1$;
therefore, 
$$\dim (f'(U_t)\cap f'(U_j))\ge k+2 -(j-t)\ge k+2-(k-1)=3$$
which gives a contradiction again. 

So, $f'(U)$ and $f'(V)$ both belong to ${\mathcal G}_{k+2}(H)$.
Since for every $U\in {\mathcal G}_{k+2}(H)$ there is  $V\in {\mathcal G}_{k+2}(H)$
such that $U,V$ are compatible and $\dim(U\cap V)=1$, the proof is complete.
\end{proof}

\subsection{The case $i>3$}
In this case, we have $m\ge i+1$ by assumption.
As in the previous subsection, we show that $f$ is adjacency preserving.

Suppose that $X,Y\in {\mathcal G}_{k}(H)$ are adjacent. Then $\dim (X+Y)=k+1$ and 
$$\dim(X+Y)^{\perp}=n-(k+1)=mk+i-k-1=(m-1)k +i-1.$$
Without loss of generality, we can assume that $m=i+1$.
In the case when $m-i-1>0$, we choose mutually orthogonal $k$-dimensional subspaces
$$X_{1},\dots,X_{m-i-1}\subset (X+Y)^{\perp}$$
(if $m-i-1=1$, then we take any $k$-dimensional subspace $X_1$ orthogonal to $X+Y$),
consider the subspaces
$$M= (X_{1}+\dots+X_{m-i-1})^{\perp}\;\mbox{ and }\; N=(f(X_{1})+\dots+f(X_{m-i-1}))^{\perp}$$
whose dimension is equal to $(i+1)k+i$ and observe that $f$ sends any $k$-dimensional subspace of $M$ to a $k$-dimensional subspace of $N$.

Let $m=i+1$.
Two $k$-dimensional subspaces of $H$  are adjacent if and only if 
their orthogonal complements are adjacent.
In particular, we have
$$\dim(X^{\perp}\cap Y^{\perp})=n-k-1=(i+1)k+i-k-1=ik +i-1.$$
Assume that $f(X)$ and $f(Y)$ are not adjacent.
Then their orthogonal complements also are not adjacent
and
$$\dim(f(X)^{\perp}\cap f(Y)^{\perp})<ik +i-1.$$
We set 
$$M_1=X^{\perp}\cap Y^{\perp}\;\mbox{ and }\;N_1=f(X)^{\perp}\cap f(Y)^{\perp}.$$
Then $f$ sends any $k$-dimensional subspace of $M_1$ to a $k$-dimensional subspace of $N_1$,
i.e. it induces an orthogonality preserving injection
$$f_{1}:{\mathcal G}_{k}(M_{1})\to {\mathcal G}_{k}(N_1),$$
where 
$$\dim N_{1}< \dim M_{1}=ik +i-1.$$
Note that $f_1$ is not adjacency preserving (otherwise, 
it is induced by a unitary or anti-unitary operator which contradicts the fact that $\dim N_1 < \dim M_1$).
Let us take any adjacent $U,V\in {\mathcal G}_{k}(M_1)$ 
such that $f(U), f(V)$ are not adjacent.
Consider the subspace
$$M_2=U^{\perp}\cap V^{\perp}\cap M_{1}$$
whose dimension is equal to
$$\dim M_1 - k-1=(i-1)k+i-2.$$
The map $f$ sends any $k$-dimensional subspace of $M_{2}$
to a $k$-dimensional subspace contained in
$$N_2=f(U)^{\perp}\cap f(V)^{\perp}\cap N_{1}.$$
Since $\dim N_{1}<\dim M_1$ and $f(U),f(V)$ are distinct non-adjacent $k$-dimensional subspaces of $N_1$,
we have $\dim N_{2}<\dim M_2$. 
As above, the restriction of $f$ to ${\mathcal G}_{k}(M_2)$ is not adjacency preserving.

Recursively, we establish that $f$ induces a sequence of maps 
$$f_{j}:{\mathcal G}_{k}(M_{j})\to {\mathcal G}_{k}(N_j),\;\;\;\;\; j=1,\dots,i-3,$$
where 
$$\dim N_{j}< \dim M_{j}=(i-j+1)k +i-j$$
and every $f_{j}$ is an orthogonality preserving injection, but it is not adjacency preserving. 
On the other hand, 
$$\dim M_{i-3}=4k+3$$ 
and $f_{j-3}$ can be considered as an orthogonality preserving injection of 
${\mathcal G}_{k}(M_{i-3})$ to ${\mathcal G}_{k}(M')$,
where $M'$ is a $(4k+3)$-dimensional complex Hilbert space containing $N_{i-3}$.
By the arguments from the previous subsection, $f_{j-3}$ is adjacency preserving.
We come to a contradiction which implies that $f(X)$ and $f(Y)$ are adjacent. 

\section{Proof of Corollary \ref{cor}}
We will use Faure-Fr\"{o}licher-Havlicek's version of the Fundamental Theorem of Projective Geometry
\cite{FaureFrolicher, Havlicek} to prove the statement for $k=1$.
Let $V$ and $V'$ be left vector spaces over division rings $R$ and $R'$, respectively.
The dimensions of the vector spaces are assumed to be not less than $3$.
Denote by ${\mathcal G}_{1}(V)$ and ${\mathcal G}_{1}(V')$ 
the sets of $1$-dimensional subspaces of $V$ and $V'$, respectively. 
A map $L:V\to V'$ is {\it semilinear} if 
$$L(x+y)=L(x)+L(y)$$
for all $x,y\in V$
and there is a non-zero homomorphism $\sigma:R\to R'$ such that 
$$L(ax)=\sigma(a)L(x)$$
for every $a\in R$ and $x\in V$.
Every semilinear injection $L:V\to V'$ induces a map between ${\mathcal G}_1(V)$ and ${\mathcal G}_{1}(V')$
which sends every $P\in{\mathcal G}_1(V)$ to the $1$-dimensional subspace containing $L(P)$.

\begin{theorem}[\cite{FaureFrolicher, Havlicek}]\label{FTPG}
Suppose that $f:{\mathcal G}_{1}(V)\to {\mathcal G}_{1}(V')$ is an injection 
satisfying the following conditions:
\begin{enumerate}
\item[(1)] if $X,Y,Z\in {\mathcal G}_{1}(V)$ and $Z\subset X+Y$, then
$f(Z)\subset f(X)+f(Y)$;
\item[(2)] there is no $2$-dimensional subspace of $V'$ which contains all elements 
from the image of $f$.
\end{enumerate}
Then $f$ is induced by a semilinear injection of $V$ to $V'$.
\end{theorem}

\begin{lemma}[Proposition 4.2 in \cite{Pankov-book}]\label{lemma0}
If an injective semilinear transformation of $H$ sends orthogonal vectors to orthogonal vectors, 
then it is a scalar multiple of a linear or conjugate-linear isometry.
\end{lemma}

Suppose that $H$ is infinite-dimensional
and $f$ is an injective transformation of ${\mathcal G}_{k}(H)$ satisfying the condition (*) 
from Corollary \ref{cor}, i.e. for any orthogonal $X,Y\in {\mathcal G}_{k}(H)$ there  is a family 
$\{Z_{j}\}_{j\in J}$ of mutually orthogonal elements of ${\mathcal G}_{k}(H)$ such that 
each $Z_j$ is orthogonal to $X+Y$ and
$H'$ is the orthogonal sum of $f(X),f(Y)$ and all $f(Z_j)$.
Recall that $H'$ is the smallest closed subspace containing all elements from the image of $f$.

Suppose that $k=1$ and show that $f$ satisfies the conditions of Theorem \ref{FTPG}.
Since $f$ is orthogonality preserving, $H'$ is infinite-dimensional and (2) holds.
For any distinct $X,Y\in {\mathcal G}_{1}(H)$ there is $Y'\in {\mathcal G}_{1}(H)$
orthogonal to $X$ and such that $X+Y=X+Y'$.
By (*), there is a family 
$\{Z_{j}\}_{j\in J}$ of mutually orthogonal elements of ${\mathcal G}_{1}(H)$ such that 
each $Z_j$ is orthogonal to $X+Y'$ and
$H'$ is the orthogonal sum of $f(X),f(Y')$ and all $f(Z_j)$.
Since $Y\subset X+Y'$ is orthogonal to all $Z_j$, 
$f(Y)$ is orthogonal to all $f(Z_j)$.
This means that $f(Y)$ is contained in $f(X)+f(Y')$
(since $H'$ is  the orthogonal sum of $f(X),f(Y')$ and all $f(Z_j)$).
By the injectivity of $f$, we have
$$f(X)+f(Y)=f(X)+f(Y').$$
Similarly, for every $1$-dimensional subspace $Z$ contained in $X+Y=X+Y'$
we establish that $f(Z)\subset f(X)+f(Y)$, i.e. the condition (1) also is satisfied.
So, $f$ is induced by a semilinear injection and Lemma \ref{lemma0} gives the claim.

In the case when $k\ge 2$, it is sufficient to show that $f$ is adjacency preserving and apply Theorem \ref{P1}.
Suppose that $X,Y\in {\mathcal G}_{k}(H)$ are adjacent.
There exists $Y'\in {\mathcal G}_{k}(H)$ orthogonal to $X$ and such that 
$$X+Y\subset X+Y'.$$
Let $\{Z_{j}\}_{j\in J}$ be a family of mutually orthogonal elements of ${\mathcal G}_{k}(H)$ such that 
each $Z_j$ is orthogonal to $X+Y'$ and
$H'$ is the orthogonal sum of $f(X),f(Y')$ and all $f(Z_j)$. 
We fix $j_{0}\in J$ and consider the $(3k)$-dimensional subspaces 
$$M=X+Y'+Z_{j_{0}}\;\mbox{ and }\; N=f(X)+f(Y')+f(Z_{j_{0}}).$$
Every $k$-dimensional subspace $Z\subset M$ is orthofonal to all $Z_{j}$ with $j\ne j_{0}$.
Then $f(Z)$ is orthogonal  to all $f(Z_{j})$ with $j\ne j_{0}$,
i.e. $f(Z)$ is contained in $N$ (since $H'$ is the orthogonal sum of $N$ and all $f(Z_j)$ with $j\ne j_{0}$).
Therefore, $f$ induces an orthogonality preserving 
injection of ${\mathcal G}_{k}(M)$ to ${\mathcal G}_{k}(N)$. 
By Theorem \ref{theorem1}, this restriction is induced by a unitary or anti-unitary operator from $M$ to $N$;
in particular, it is adjacency preserving and $f(X),f(Y)$ are adjacent.

\subsection*{Acknowledgment}
The author is grateful to Hans Havlicek for useful remarks.

\end{document}